\documentclass[a4paper,12pt]{article}
\usepackage{amsmath,mathrsfs,amsthm,amsfonts,setspace}
\usepackage{amssymb}
\usepackage{enumerate}
\usepackage{graphicx,epsfig,psfrag,float}

\newcommand{\mb}{\mathbb}

\newtheorem{theorem}{Theorem}

\newtheorem{corollary}{Corollary}
\newtheorem{definition}{Definition}

\newtheorem{proposition}{Proposition}
\begin{document}

\begin{center}
\textbf{Weyl-Heisenberg frame operators and Kohn-Nirenberg symbols}\\*[4mm]
\textrm{T.C.Easwaran Nambudiri}\\ Department\ of\ Mathematics,\ Government\ Brennen \ College, \\
\ Dharmadam, \ Thalassery,  Kerala \ 670106,\ India. \\
              Tel.: +91 9497384926\\
              easwarantc@gmail.com \\*[2mm]
\textrm{K.Parthasarathy}\\
Ramanujan \ Institute \ for \ Advanced \ Study \ in \ Mathematics\\
University \ of \ Madras, \ Chennai \ 600005, \ India.\\
krishnanp.sarathy@gmail.com \\*[2mm]


\end{center}

\begin{abstract}
An explicit expression for the Kohn-Nirenberg symbol of a Weyl- Heisenberg frame operator on $L^2(\mathbb{R})$ is obtained directly from the Gabor atom coming from new classes of window functions.
This new approach, using only elementary Fourier analysis, is independent of the theory of distributions and works strictly inside $L^2(\mathbb{R})$. Kohn-Nirenberg operators are introduced and are shown to be Weyl-Heisenberg frame operators in suitable cases.
\end{abstract}

\section{Introduction}
\label{intro}
	Time-frequency analysis exploits  translations and  modulations to analyse functions and operators. Gabor analysis is the outcome of the confluence of time-frequency analysis and the theory of Hilbert space frames (\cite{DS}). Janssen's work (\cite{J}) initiated its mathematical investigations and \cite{DGM} marked its emergence as an important research area.  A central object in the theory, both from the theoretical and applications points of view, is the frame operator. Frame operators of Weyl-Heisenberg frames in $L^2(\mathbb{R})$ have been completely characterised  (\cite{EP}). We seek to get a better insight about these operators by viewing it as an integral operator from $L^2(\mathbb{R})$ into $L^2(\mathbb{R})$ itself, rather than as a map from modulation spaces into the space of tempered distributions.

 A large quantum of work has been carried out by experts using abstract theories in very general settings (see for instance,\cite{BS}, \cite{HGFK}, \cite{W}, Chapters 11 and 14 of \cite{G}). Most of the known results are about such (pseudo-differential) operators, mapping a restrictive class like the Schwartz space into a space, more general (e.g.the space of tempered distributions \cite{HGFK}) than what is actually required, whereas Weyl-Heisenberg frame operators are maps from $L^2(\mathbb{R})$ into itself. As pointed out in \cite{G} (Chapter 14) and \cite{G1}, very little is known about the boundedness of these operators when their Gabor atoms lie outside the modulation space $M^1$ or the Wiener space $W$. The recent survey \cite{G1} on the intrigues of Gabor frames mentions the need for fresh approaches and new classes of window functions to tackle a number of fundamental open problems in the field.

Here our aim is modest:
  \textit{identify some specific function spaces in $L^2(\mathbb{R})$ as suitable classes for Gabor atoms and obtain the Kohn-Nirenberg symbol of the associated frame operator directly from the Gabor atom, in an elementary fashion, without bringing in any abstract theory.}

Although the role of pseudo-differential operators in Gabor analysis (\cite H) and the representation of the Weyl-Heisenberg frame operators using Gabor multipliers (\cite {DT})  have been discussed before, an explicit expression for the Kohn-Nirenberg symbol (\cite{K}) of a Weyl-Heisenberg frame operator in terms of its Gabor atom is not seen in the literature. We provide this through a direct approach, based only on elementary Fourier analysis. New classes $\mathcal E_{a,b}$ and $\mathcal  P_{a,b}$ of window functions in $L^2(\mathbb{R})$ are introduced for this purpose. Our symbol theorem holds for Weyl-Heisenberg frames having Gabor atoms in the larger class $\mathcal P_{a,b}$ and leads to Kohn-Nirenberg operators, which turn out to be Weyl-Heisenberg frame operators under suitable conditions.

 Some needed definitions and facts about abstract frames, frame operators and Weyl-Heisenberg frames are given in section 2. New function spaces and the symbol function are introduced in Section 3. The symbol theorem, Kohn-Nirenberg operators and some applications are presented in the last section.
\section{Preliminaries}
\label{sec:1}
A family $\lbrace u_{k} : k \in \mathbb{N}\rbrace$  in a Hilbert space $\mathcal {H}$ is called a \textit{frame}, if the inequality:
$$\alpha \|x \|^{2} \leq \Sigma_k |\langle x, u_{k}\rangle|^{2} \leq \beta \|x \|^{2}$$
 holds for some positive constants $\alpha$ and $\beta$ and for all $x \in \mathcal {H}$.
 The \textit{frame operator} of a frame is given by $Sx = \Sigma \langle x, u_{k}\rangle u_{k}, \  x \in \mathcal {H}$,
\noindent the series converging unconditionally, and is a bounded linear, positive, invertible operator on $\mathcal{H}$. If only the upper inequality is satisfied, $\{u_k\}$ is called a \textit{Bessel sequence} and the operator $S$ is still defined as a bounded linear operator. We call it the \textit{preframe operator} of $\{u_k\}$. 

Here we only consider  \textit{Weyl-Heisenberg frames} (also known as \textit{Gabor frames}), a special class of frames of the form  $(g,a,b) := \lbrace E_{mb}T_{na}g: m,n \in \mathbb Z\rbrace$ in $L^2(\mathbb R)$, generated by translations $T_{na}$ and modulations $E_{mb}, a,b>0$ of a $g\in L^2(\mathbb R)$ (known as a \textit{Gabor atom} or a \textit{window function}). 

The Fourier transform  $\widehat f$ of an $f\in L^1(\mathbb R)$ is the function defined on $\mathbb R$ by 
$$\widehat f(\xi) = \int_{\mathbb R} \ f(t) \ e^{-2\pi \imath \xi t}\ dt, \ \xi\in \mathbb R.$$ If $f\in L^1 \cap L^2(\mathbb R)$, then $\widehat f \in L^2(\mathbb R)$ with $\|f\|_2 = \|\widehat f\|_2$ and the Fourier transform extends to a unitary operator $\mathcal F$ on $L^2(\mathbb R)$. We call $\mathcal F$ the Fourier transform operator on $L^2(\mathbb R)$ and write, for notational convenience, $\widehat g$ for $\mathcal Fg$ and $\check g$ for $\mathcal F^{-1}g$ even when $g\in L^2(\mathbb R)$.

The background material can be found in \cite{C} and \cite{G}.
\section{Function spaces for generating symbols}
The subspaces of $L^2(\mathbb R)$ introduced here will form the setting for the construction of our explicit expression for the Kohn-Nirenberg symbol for a Weyl-Heisenberg frame operator.

\begin{proposition}
For $g \in L^{2}(\mathbb{R})$, the series $\Sigma_n\ g(x-na) \overline{g}(x+t-na)$ converges absolutely for almost every $x,t \in \mathbb{R}$ and any $a>0$.
\end{proposition}

\begin{proof}
Use Schwarz inequality and the fact
$\|g\|_2^{2} = \underset {0}{\overset{a}\int} (\Sigma_n  |g(x-na)|^{2})\ dx$.
\end{proof}

\begin{definition}
\rm{For $a>0$ and $g \in L^{2}(\mathbb{R})$, the associated function $\Phi$ is defined by $$\Phi(x,t):= \Sigma_n \ g(x-na) \overline{g}(x+t-na),\ a.e. \ x,t\in \mathbb {R}$$}.
\end{definition}

\begin{proposition}
Suppose $g\in L^{2}(\mathbb{R})$ satisfies $\Sigma_n |g(x-na)| < \infty,  a.e. x\in \mathbb R$. Then 

i) $\Phi_{x} \in L^{2}(\mathbb{R})$ for a.e. $x \in \mathbb{R}$ and

ii) $\check \Phi_{x} \in L^{2}\cap L^{1}(\mathbb{R})$ for a.e. $x \in \mathbb{R}$ when $ \widehat{g} \in L^{1}(\mathbb{R})$, where $\Phi_{x}(t) :=  \Phi(x,t)$.
\end{proposition}
\begin{proof}
Fix an $x\in \mathbb{R}$ such that $\Sigma_n |g(x-na)| < \infty $. Then the partial sums $S_{k}$ defined by
$S_{k} = \Sigma_{n=-k}^k \ g(x-na)$ $\overline{T_{na-x}g}$,
form a Cauchy sequence in $L^{2}(\mathbb{R})$,  and the limit is just $\Phi_{x}$. Hence $\Phi_{x} \in L^{2}(\mathbb{R})$ and $\check \Phi_{x} \in L^{2}(\mathbb{R})$ for a.e. $x \in \mathbb{R}$. 

For ii), use similar arguments for $\check{S_{k}}$ show that $\check \Phi_{x} \in L^{1}(\mathbb{R})$ for a.e. $x \in \mathbb{R}.$
\end{proof}

Motivated by this, we now introduce our function spaces $\mathcal P_{a,b}$ and $\mathcal E_{a,b}$.

\begin{definition}
\rm{For $a,b>0$ let $\mathcal P_{a,b}$ be the space
of those $g\in L^1(\mathbb R)$ satisfying

i) $\Sigma_m |\check \Phi_{x} (\xi - mb)| \leq B_x$ for some $B_x$, for a.e. $\xi, x \in \mathbb{R}$; \ \  ii) $\widehat g\in L^1(\mathbb R)$.

The space $\mathcal E_{a,b}$ is the class of functions $g\in L^1(\mathbb R)$ for which there are positive constants $A,B$
such that $ \Sigma_n |g(x - na)| \leq A$ and $
\Sigma_m | \widehat{g}(\xi - mb)| \leq B$ for a.e. $x,\xi \in \mathbb R$.

 For $g \in \mathcal P_{a,b}$, define $\Psi$ by
$\Psi(x, \xi):= \Sigma_m \ \check \Phi_{x} (\xi - mb)$, a.e.  $x, \xi \in \mathbb{R}.$}
\end{definition}
\begin{proposition}\label{PE}
For all $a,b>0$, $\mathcal {E}_{a,b}$ is a subspace of $\mathcal P_{a,b}$ that is invariant under both translations and modulations.
\end{proposition}
\begin{proof}
Since $\widehat{T_cf} = E_{-c}\widehat{f}$ and $\widehat{E_cf} = T_{c}\widehat{f}$ for $c\in \mathbb{R}$, we need only to establish the inclusion: $\mathcal {E}_{a,b}\subset \mathcal {P}_{a,b}$. For $g\in \mathcal {E}_{a,b}$, $\Sigma_n|g(x-na)| \int_{\mathbb{R}} |\overline{g}(x+t-na)|\ dt$ is finite for a.e. $x$ and an application of the dominated convergence theorem yields

$\check \Phi_{x}(\xi) \ = \Sigma_n\ g(x-na)\ e^{-2\pi i\xi (x-na)} \int_{\mathbb{R}}\overline{g}(u) \ e^{2\pi i\xi u}du$
\\\hspace*{1.5cm}
$= \check{\overline{g}}(\xi) \Sigma_n\ g(x-na)\ e^{-2\pi i\xi (x-na)}.$

\noindent This leads to an estimate stronger than asserted, independent of $x$:

$\Sigma_m\ | \check \Phi_{x} (\xi - mb)|
\leq \Sigma_m \ | \check{\overline{g}}(\xi - mb)| \Sigma_n\ |g(x-na)|\leq AB$.
\end{proof}

The Wiener space $W\subset L^1 \cap L^2(\mathbb{R})$ is the space of measurable functions $g$  with $\|g\|_W := \Sigma_k \|g\chi_{[k, k+1)}\|_{\infty} < \infty.$

According to experts, $W \cap \widehat W$ is a natural and practically important space for sampling (See \cite{G} for details on $W$ and its relation to sampling.) Thus, the following inclusions are of interest.

\begin{proposition}
For all $a, b>0$, the space $\mathcal {E}_{a,b}$ contains $ W\cap \widehat{W}$ and hence the Schwartz space $\mathcal S$ as well as the Feichtinger algebra $\mathcal S_0$.
\end{proposition}

\begin{proof}
For $g \in W\cap \widehat{W}$, there is a constant $C_a$ for any $a>0$ such that
 \\\hspace*{2cm}$\Sigma_n |g(x-na)| \leq C_a \|g\|_W$ 
\\ for a.e. $x\in \mathbb R$  ( \cite{C} p.221, \cite{G} p.105). 
The first assertion follows since if $g = \widehat h, h\in W$, then $\widehat g(\xi) = h(-\xi)$ so that $\widehat{g} \in W$ and so, for any $b>0$ and a.e. $\xi \in \mathbb R$,
 \\\hspace*{2cm}$\Sigma_m |\widehat g(\xi - mb)| \leq C_b \|\widehat{g}\|_W $.

Next we observe that $g\in W\cap \widehat{W}$ if both $g(t)$ and $\widehat g(t)$ are $ O(1/(1+|t|)^2)$.
\\
For, if $C > 0$ is such that $ |g(t)| \leq C/(1+|t|)^2$
for all $t \in \mathbb{R}$, then
\\\hspace*{2cm}
$\gamma_n:= \max \{|g(t)|: n\leq t\leq n+1\} \leq C/(1+|n|)^2, n\geq 0$
\\\noindent and similarly $\gamma_n \leq C/(1+|n-1|)^2, n\leq 0$. Thus $\Sigma_n \gamma_n <\infty$ and $g\in W$.
 
 The same considerations for $\widehat g$ in place of $g$ gives $\widehat g\in W$ and so $g\in \widehat{W}$. 
 
 In particular, $\mathcal S$ lies in $ W\cap \widehat{W}$.
Finally, the Feichtinger algebra $\mathcal S_0$ (or the modulation space $M^1$) also lies in $W \cap \widehat{W}$, by Proposition 12.1.4 of \cite{G}.
\end{proof}
Among the many interesting properties of  spaces $\mathcal E_{a,b}$ and $\mathcal P_{a,b}$, we content ourselves with presenting only those that are relevant to the symbol function.

\begin{proposition} \label{FSP}
For $g \in \mathcal P_{a,b}$, $\Psi$ has an absolutely convergent expansion
\\\hspace*{2cm}
 $\Psi(x, \xi) = \Sigma_m
 \Sigma_n \ g(x-na)\
  \check{\overline{g}}(\xi - mb)  e^{-2\pi i(x-na)(\xi - mb)}$
\\\indent Further, for almost every $x,$ the function $\Psi_x$ is integrable
on $[0, b)$ and has an absolutely convergent Fourier series expansion
\\\hspace*{3cm}  $ \Psi_{x}(\xi) = \Sigma_k \ c_{k}e^{2\pi i\xi (k/b)}$
 \\where $c_{k} = c_k(x) = (1/b)\Sigma_n \ g(x-na) \overline{g}(x-na+k/b).$
\end{proposition}
\begin{proof}
First note that $\check \Phi_{x}\in L^{2}\cap L^1(\mathbb{R})$ by the previous proposition. 
\\
Moreover $\Sigma |g(x-na)| \int_{\mathbb{R}} |\overline{g}(x+t-na)|\ dt$ is finite for a.e. $x$ and an application of the dominated convergence theorem yields expansion for $\check \Phi_{x}$, hence for $\Psi$:

$\check \Phi_{x}(\xi)  = \int_{\mathbb{R}}\Sigma \ g(x-na)\overline{g}(x+t-na) \ e^{2\pi i\xi t}dt$
\\\hspace*{1.4cm}
$= \Sigma \ g(x-na) \int_{\mathbb{R}}\overline{g}(x+t-na) \ e^{2\pi i\xi t}dt$
\\\hspace*{1.4cm}
$= \Sigma \ g(x-na)\ e^{-2\pi i\xi (x-na)} \int_{\mathbb{R}}\overline{g}(u) \ e^{2\pi i\xi u}du$
\\\hspace*{1.4cm}
$= \check{\overline{g}}(\xi) \Sigma \ g(x-na)\ e^{-2\pi i\xi (x-na)}$,

$\Psi(x, \xi) = \Sigma_m  \check \Phi_{x} (\xi - mb)$
$= \Sigma_m \Sigma_n g(x-na) \ \check{\overline{g}}(\xi - mb)  e^{-2\pi i(x-na)(\xi - mb)},$
\\ the absolute convergence of the series being a consequence of the assumption on $g$. Writing $A_x = \Sigma \ |g(x-na)|$, from this we get the estimates:

$\Sigma\ | \check \Phi_{x} (\xi - mb)|
 \leq A_x \  \Sigma \ | \check{\overline{g}}(\xi - mb)|$
 as well as

 $\int_{[0, b)} |\Psi_x(\xi)| \ d\xi \leq A_x \int_{[0, b)}
 \Sigma \ | \check{\overline{g}}(\xi - mb)| \ d\xi $
$= A_x \int_{\mathbb R} | \check{\overline{g}}(\xi)| \ d\xi$
$ <  \infty,$
\\the last integral being finite because
$\check{\overline{g}}\in L^1(\mathbb R)$.
The Fourier coefficients $c_{k}$ of $\Psi_{x}$ can now be evaluated without difficulty:
\\\hspace*{6mm} 
$c_{k}  = (1/b) \int_{[0, b)} \Psi_{x}(\xi) e^{-2\pi i\xi (k/b)}d\xi$
\\\hspace*{1cm} $=(1/b) \int_{[0,b)} \Sigma_m \ \check{\overline{g}}
(\xi - mb) \Sigma_n g(x-na) e^{-2\pi i(\xi-mb)(x-na)} e^{-2\pi i\xi(k/b)} d\xi$
\\\hspace*{1cm} $= (1/b)\Sigma_n \ g(x-na) \int_{[0,b)}
\Sigma_m\ \check{\overline{g}}(\xi - mb)e^{-2\pi i[(\xi-mb)(x-na)+\xi(k/b)]}d\xi $
\\\hspace*{1cm} $=(1/b)\Sigma_n\ g(x-na)
\int_{\mathbb{R}}\check{\overline{g}}(v)e^{-2\pi iv(x-na+k/b)}dv $
\\\hspace*{1cm} $=(1/b)\Sigma_n\ g(x-na)\overline{g}(x-na+k/b)$
\\ since the integral in the penultimate step exists and is
$\mathcal F \mathcal F^{-1}{\overline{g}}(x-na+k/b) = \overline{g}(x-na+k/b).$
The absolute convergence of the double series for $\Psi$ justifies the interchange of summations over $m$ and $n$. Analogous reasoning, using Fubini, validate taking the integral inside the summation over $n$.

 To see absolute convergence, observe that

$ \Sigma_k\ |c_{k}| \leq  (1/b)\Sigma_k \ \Sigma_n\ |g(x-na)||\overline{g}((x-na)+(k/b))|$
\\\hspace*{1.8cm} $ = (1/b) \Sigma_n \ |g(x-na)|( \Sigma_k\ |\overline{g}((x-na)+(k/b))|)$
 $ < \infty, $
\\the last two sums being finite because $g\in L^1(\mathbb R)$.
\end{proof}
In the literature, $\Sigma_n \ g(x-na)\overline{g}(x-na+k/b)$ is usually
denoted by $G_k(x)$ for $k\in \mathbb Z$. 
In our notation,
$G_k(x) = \Phi_x(k/b) = b\ c_k(x)$. Thus
$ \Psi_{x}(\xi) = (1/b) \Sigma_k G_k(x)\ e^{2\pi i\xi (k/b)}$.
\section{Kohn-Nirenberg symbols and operators}
 Now we express a Weyl-Heisenberg frame operator in terms of the Kohn-Nirenberg symbol. This leads to Kohn-Nirenberg operators. We make use of the dense subspace $A_1(\mathbb R) := \lbrace f\in L^1(\mathbb{R}): \widehat f \in L^1(\mathbb R)\rbrace$ of $L^2(\mathbb R)$.
We adopt the following definition from \cite{G} for our symbol theorem. 

A  \textit{pseudo-differential operator with Kohn-Nirenberg symbol $\sigma$} is an operator of the form
$ K_{\sigma}f(x) := \int_{\mathbb{R}} \sigma (x, \xi) \widehat{f}(\xi) e^{2 \pi \imath x \xi} d \xi$.
\begin{theorem}
Let $g \in \mathcal P_{a,b}$ and suppose that $(g, a, b )$ is a Bessel sequence. Then its preframe operator $S$ is given by
\\\hspace*{2cm}$Sf(x) = \int_{\mathbb{R}} \Psi (x, \xi) \widehat{f}(\xi) e^{2 \pi \imath x \xi} d \xi = \mathcal {F}^{-1}M _{\Psi_{x}}\mathcal {F}f(x)$  \\for almost every $x \in \mathbb{R}$ and for all $f$ in the dense subspace $A_1(\mathbb R)$. Thus, on $A_1(\mathbb R)$, $S$ is the pseudo-differential operator with Kohn-Nirenberg symbol $\Psi$.
\end{theorem}
\begin{proof}
For convenience, we write $e(x)$ for $e^{2\pi \imath x}$ in this proof.
Since $g \in \mathcal P_{a,b}$, we have $\Sigma_m \ |\check \Phi_{x}(\xi - mb)| \leq B_x < \infty $
for some $B_x>0$ and  a.e. $x,\xi \in \mathbb{R}$. Fix such an $x$ and consider the bounded measurable function $\Psi_{x}$
associated with the triplet $(g,a,b).$
For $f \in A_1(\mathbb{R})$ we have $\mathcal {F}f = \widehat{f}$ and

 $\int_{\mathbb{R}} \Sigma_m\ | \check \Phi_{x} (\xi-mb)|
|\widehat{f}(\xi)| d \xi \ \leq B_x \| \widehat{f}\|_1  < \infty$.
\\ Thus an application of dominated convergence theorem is valid and yields

$\mathcal {F}^{-1}M _{\Psi_{x}}\mathcal {F}f(x)
= \int_{\mathbb{R}}(\Sigma_m \check \Phi_{x}(\xi -mb))\widehat{f}(\xi)e(\xi x)\ d\xi$
 \\\hspace*{2cm}
$= \Sigma_m\ \int_ {\mathbb{R}} \check \Phi_{x} (\xi -mb) \widehat{f}(\xi)e(\xi x)\  d\xi$
 \\\hspace*{2cm} $= \Sigma_m\ \int_ {\mathbb{R}} \check \Phi_{x} (\xi -mb)
(\int_{\mathbb{R}} f(t)e(-\xi t) dt) e(\xi x)\ d\xi.$ (*)

Now $\int_ {\mathbb{R}}| \check \Phi_{x} (\xi -mb)|\ d\xi \int_{\mathbb{R}} |f(t)|\ dt $ $= \|\check \Phi_{x}\|_{1} \|f\|_{1} < \infty$,

\noindent so an application of Fubini's theorem below is justified and we compute:

$ \int _ {\mathbb{R}}\check \Phi_{x} (\xi -mb)\int_{\mathbb{R}} f(t) e(-\xi t)\ dt\
 e(\xi x) d \xi$
 \\\hspace*{2cm} 
= $\int_{\mathbb{R}} f(t) \int_ {\mathbb{R}} \check \Phi_{x}(\xi-mb) e(\xi(x-t))  d\xi \ dt$ 
\\\hspace*{2cm} 
= $\int_{\mathbb{R}} f(t) \int_ {\mathbb{R}} \check \Phi_{x}(u)e((u+mb)(x-t))\ du\ dt$
\\\hspace*{2cm} 
= $\int_{\mathbb{R}} f(t)\int_ {\mathbb{R}} \check \Phi_{x}(u) e(-u(t-x))\ du\ e(mb(x-t))\ dt$ 
\\\hspace*{2cm} 
= $\int_{\mathbb{R}} f(t) \ \mathcal F \mathcal F^{-1} \Phi_x (t-x)\ e(mb(x-t)) \ dt$
\\\hspace*{2cm} 
= $\int_{\mathbb{R}} f(t) \Phi_{x} (t-x) e(mb(x-t))\ dt$
 \\\hspace*{2cm} 
= $e(mbx)\int_{\mathbb{R}}f(t) \Phi_{x} (t-x) e(-mbt)\ dt$
 \\\hspace*{2cm} 
= $e(mbx)\int_{\mathbb{R}} f(t) \Sigma_n g(x-na)\overline{g} (t-na) e(-mbt)\ dt$
 \\\hspace*{2cm} 
= $e(mbx) \int_{\mathbb{R}} f(t) \Sigma_n \ T_{na}g(x)\overline{T_{na}g} (t) e(-mbt)\ dt$.

\noindent But $|f|, | T_{na}\overline{g}| \in L^{2}(\mathbb{R})$ so we have
$\int_{\mathbb{R}}|f(t)| | \overline{T_{na}g}(t)|\ dt < \infty$ by Schwarz inequality and consequently
$\Sigma_n \ |T_{na}g(x)| \int_{\mathbb{R}} |f(t)| | \overline{T_{na}g}(t)|\ dt < \infty$.
This validates an application of dominated
convergence theorem, and we get

$\int _ {\mathbb{R}}\check \Phi_{x} (\xi -mb)( \int_{\mathbb{R}} f(t) e(-\xi t) dt)
e(\xi x)\ d\xi$
\\\hspace*{2cm} 
= $e(mbx)\Sigma_n \ T_{na}g(x) \int_{\mathbb{R}}f(t) e(-mbt) \overline{T_{na}g}(t)\ dt$
\\\hspace*{2cm} 
= $e(mbx) \Sigma_n \ T_{na}g(x) \int_{\mathbb{R}}f(t) \overline{E_{mb}T_{na}g}(t)\ dt$
\\\vspace*{2mm}\hspace*{2cm} 
= $\Sigma_n e(mbx) \ T_{na}g(x) \langle f, E_{mb} T_{na}g \rangle$ 
\\\vspace*{2mm}\hspace*{2cm} 
= $\Sigma_n E_{mb}\ T_{na}g(x) \langle f, E_{mb} T_{na}g \rangle $
\\\vspace*{2mm}\hspace*{2cm} 
= $\Sigma_n \langle f, E_{mb} T_{na}g \rangle E_{mb}\ T_{na}g(x).$

\noindent Substituting the expression on the right in (*), we thus get
\\\hspace*{2cm} 
$\mathcal {F}^{-1}M _{\Psi_{x}}\mathcal {F}f(x)$
$= \Sigma_{m,n} \langle f, E_{mb} T_{na}g \rangle E_{mb}\ T_{na}g(x)$\\
for all $x$ in a set $E_1$ of full measure. 
\\But  $\Sigma_{m,n} \langle f, E_{mb} T_{na}g \rangle E_{mb}\ T_{na}g = Sf$
on a set $E_2$ of full measure. 

Thus $\mathcal {F}^{-1}M _{\Psi_{x}}\mathcal {F}f(x) = Sf(x)$ for all $x$
in the set $E_1\cap E_2$ of full measure, thereby completing the proof.
\end{proof}
Motivated by the symbol theorem above, we define the \textit{Kohn-Nirenberg operator}  $K_{\Psi}$, corresponding to the symbol function $\Psi$ for a $g\in \mathcal P_{a,b}$, by
 \\\hspace*{2cm} $ K_{\Psi} (f)(x) = \mathcal F^{-1}M_{\Psi_x} \mathcal{F}f(x)$
\\\noindent for $f\in A_1(\mathbb{R})$ and a.e. $x \in \mathbb{R}$.

An important problem for pseudo-differential operators is their $L^2$ boundedness. We  find situations when a Kohn-Nirenberg operator is a bounded linear operator on $L^2(\mathbb{R})$ and yields a Weyl-Heisenberg preframe operator.
\begin{theorem} \label{KN1}
 Suppose $g \in \mathcal P_{a,b}$ satisfies any one of the following conditions:

 i) $\underset{m,n}{\Sigma} \ | \langle f, E_{mb}T_{na}g\rangle | \leq \beta \|f\|_2$ for some $\beta >0$ and for all $f \in A_1(\mathbb{R})$;

ii) $\underset{n}{\Sigma} |g(x-na)| \leq A$ and $\underset{k}{\Sigma} |g(x-k/b)| \leq B$ for a.e. $x\in \mathbb{R}$ for some positive constants $A, B$ and $0< ab \leq 1$.

In each of these cases, both of the following assertions hold.

a) $K_{\Psi}$ is defined on a dense subspace $D$ of $L^2(\mathbb{R})$,
 $K_{\Psi}(f) \in L^2(\mathbb{R})$ for all $f \in D$ and $K_{\Psi}$ extends to a positive, bounded linear operator on $L^2(\mathbb{R})$.

b) $(g,a,b)$ is a Bessel sequence with preframe operator $S = K_{\Psi}$.
\end{theorem}
\begin{proof}
 We prove a) in each case and b) will follow easily from known results.
\\i) Take $D$  as the dense subspace $A_1(\mathbb{R})$ of $L^2(\mathbb{R})$. Since $g \in \mathcal P_{a,b}$, for all $f\in D$, as in the proof of the representation theorem, we have

\hspace{2cm}$\mathcal F^{-1}M_{\Psi_{x}}\mathcal Ff(x) = \Sigma_{m,n}\  \langle f, E_{mb}T_{na}g\rangle \ E_{mb}T_{na}g(x)$
\\for a.e. $x \in \mathbb{R}$. Using this and applying the Schwarz inequality we have
\\
\noindent $\int_{\mathbb{R}} |K_{\Psi}(f)(x)|^2 \ dx$
\\\hspace*{1cm}  $= \int_{\mathbb{R}} |\mathcal F^{-1}M_{\Psi_{x}}\mathcal{F}f(x)|^2 \ dx $
\\\hspace*{1cm}  $= \int_{\mathbb{R}} |\Sigma_{m,n} \  \langle f, E_{mb}T_{na}g\rangle \ E_{mb}T_{na}g(x)|^2 \ dx $
\\\hspace*{1cm}  $\leq \int_{\mathbb{R}} \Sigma_{m,n} \ | \langle f, E_{mb}T_{na}g\rangle| \ |E_{mb}T_{na}g(x)|
\Sigma_{k,l} \ | \langle f, E_{kb}T_{la}g\rangle| |E_{kb}T_{la}g(x)| \ dx $
\\\hspace*{1cm}  $=  \Sigma_{m,n} \ | \langle f, E_{mb}T_{na}g\rangle| \Sigma_{k,l} \ | \langle f, E_{kb}T_{la}g\rangle|
\int_{\mathbb{R}}  |E_{mb}T_{na}g(x)||E_{kb}T_{la}g(x)| \ dx $
\\\vspace*{2mm}\hspace*{1cm}  $\leq \Sigma_{m,n} \ | \langle f, E_{mb}T_{na}g\rangle| \ \Sigma_{k,l} \ |\langle f, E_{kb}T_{la}g\rangle| \|E_{mb}T_{na}g\|_2  \ \|E_{kb}T_{la}g\|_2$
\\\vspace*{2mm}\hspace*{1cm}  $= (\Sigma_{m,n} \ | \langle f, E_{mb}T_{na}g\rangle|)^2 \|g\|_2^2 $
\\\vspace*{2mm}\hspace*{1cm}  $\leq \beta^2 \|f\|_2^2 \|g\|_2^2$.

\noindent Thus $K_{\Psi}f \in L^2(\mathbb{R})$ and $K_{\Psi}$ is a bounded linear operator on $D$. By denseness of $D$, it extends to a bounded operator on $L^2(\mathbb{R})$.
Now  if $T$ is the preframe operator of the Bessel sequence $(g,a,b)$, we have, for $f\in D$,

$ \langle K_{\Psi}(f), f\rangle = \int_{\mathbb{R}} \mathcal F^{-1}M_{\Psi_x}\mathcal F f(x)\ \bar f(x)\ dx$
\\\hspace*{2.1cm} $=\int_{\mathbb{R}} \Sigma_{m,n} \  \langle f, E_{mb}T_{na}g\rangle \ E_{mb}T_{na}g(x)\ \bar f(x) \ dx$
\\\hspace*{2.1cm} $= \int_{\mathbb{R}} Tf(x) \bar{f}(x) dx$
  $= \langle Tf, f \rangle $
\\\vspace*{2mm}\hspace*{2.1cm} $= \Sigma_{m,n} \ | \langle f, E_{mb}T_{na}g\rangle|^2$
  $\geq 0.$
\\\noindent This shows that $K_{\Psi}$ is a positive operator.

ii) The assumed conditions on $g$ give the estimate
\\\vspace*{2mm}\hspace*{1cm} 
$\Sigma_k |G_k(x)| \leq \Sigma_{k,n}  |g(x-na) \bar g(x-na+k/b)|$
\\\vspace*{2mm}\hspace*{2.4cm} $= \Sigma_n |g(x-na)|  \Sigma_k |\bar g(x-na+k/b)|$
  $\leq AB,$

\noindent and this, in turn, yields the estimate
$\Sigma_k |G_k(x)|^2 \leq (\Sigma_k |G_k(x)|)^2 \leq (AB)^2.$
 In this case the dense subspace $D$ we consider is the space of compactly supported, bounded functions in $L^2(\mathbb{R})$. By Proposition 2.4 of \cite{CCJ} the series $(1/b) \Sigma_k (T_{k/b}f) G_k$ converges unconditionally in the norm of $L^2(\mathbb{R})$ and

\vspace*{2mm} $\langle (1/b) \Sigma_k (T_{k/b}f) G_k, f\rangle = \Sigma_{m,n} |\langle f, E_{mb}T_{na}g\rangle|^2, f\in D$.

\vspace*{2mm}  But
$(1/b) \Sigma_k (T_{k/b}f)(x) G_k(x) = \mathcal F^{-1}M_{\Psi_x}\mathcal F f(x)
= K_{\Psi}(f)(x)$ for a.e. $x$.

\noindent Hence $K_{\Psi}(f) \in L^2(\mathbb{R})$ for $f\in D$ and (see \cite{CCJ}, proof of Proposition 2.4)

\noindent $\|K_{\Psi}(f) \|_2 ^2 = \|(1/b) \Sigma_k (T_{k/b}f) G_k\|_2^2 $
  $\leq  \int  \ |f(x)|^2  \Sigma_k|G_k(x)|^2 \ dx$
  $\leq  (AB)^2 \|f\|_2^2$.

\noindent \vspace*{2mm}Thus $\|K_{\Psi}(f) \|_2 \leq AB \|f\|_2$ and $K_{\Psi}$is clearly linear and bounded on $D$ and so extends to the whole of $L^2(\mathbb{R})$. The operator is positive
since
\\\hspace*{2cm} 
$\langle K_{\Psi}(f), f\rangle = \Sigma_{m,n} \ | \langle f, E_{mb}T_{na}g\rangle|^2 \ \geq 0$
for $f\in D$.

\vspace*{2mm}To see b), note that  the upper frame inequality is a consequence of the assumption on $g$ in case i). In case ii), $\Sigma_k |G_k(x)|$ is bounded almost everywhere and so by a well known result of Casazza and Christensen (Theorem 9.1.5, \cite{C}), $(g,a,b)$ is a Bessel sequence. The last assertion is clear since, in both cases,

\vspace*{2mm}$\langle K_{\Psi}(f), f\rangle = \Sigma_{m,n} \ | \langle f, E_{mb}T_{na}g\rangle|^2  = \langle Sf, f\rangle$.
\end{proof}
 If either $g \in \mathcal E := \cap \{ \mathcal {E}_{a,b} : 0 < ab < 1 \}$ (note that $\mathcal  E$ itself is a large space containing the Schwartz space) or is a compactly supported bounded function in $\mathcal  P_{a,b}$ and $0<ab \leq 1$, then condition ii) of Theorem \ref{KN1} is satisfied so that $K_{\Psi}$ extends to a bounded, positive linear operator on $L^2(\mathbb{R})$ and becomes the preframe operator of the Bessel sequence $(g,a,b)$. In particular, the assertion holds if $g \in C_c^{\infty}(\mathbb{R})$. Thus for a large class of generators of Gabor frames in $L^2(\mathbb{R})$, the corresponding frame operator is the Kohn-Nirenberg operator $K_{\Psi}$, whose Kohn-Nirenberg symbol $\Psi$ is explicitly given in terms of  the Gabor atom $g$ and the frame parameters $a$ and $b$ by
$$\Psi(x, \xi) = \Sigma_m\ \Sigma_n\ g(x-na)\ \check{\overline{g}}(\xi - mb)  e^{-2\pi i(x-na)(\xi - mb)}.$$

In this expression of $\Psi$, the symmetry in time and frequency aspects, the equal importance given to both the frame parameters $a$ and $b$ and the independence of the adjoint lattice parameters $1/a$ and $1/b$ are significant.

Even if $g \in L^2(\mathbb{R})$ is not meeting the requirements of Theorem \ref{KN1}, as discussed in \cite{EP1}, it is possible to approximate the frame operator of a Weyl-Heisenberg frame $(g,a,b)$ in $ L^2(\mathbb{R})$ by preframe operators generated by window functions chosen from the class $C_c^{\infty}(\mathbb{R})$. Since $C_c^{\infty}(\mathbb{R})$ functions meet the requirements of Theorem \ref{KN1}, the corresponding preframe operators are all Kohn-Nirenberg operators. Hence in view of Lemma 5 in \cite{EP1}, we have the following: 
\begin{corollary}
For every  Weyl-Heisenberg frame operator $S$ on $L^2(\mb{R})$, there is a sequence $\{K_{\Psi_j}\}$ of Kohn-Nirenberg operators on $L^2(\mb{R})$ such that $\lim_j \langle K_{\Psi_j} f, f \rangle = \langle Sf, f \rangle$ for $f\in B_c(\mb R)$.
\end{corollary}
Analogously, other results on approximations, subsequent to Lemma 5 in \cite{EP1} can also be restated in terms of Kohn-Nirenberg operators on $L^2(\mb{R})$. 

The next corollary, showing that the operator $K_{\Psi}$ is useful for characterising Weyl-Heisenberg  frames in $L^2(\mathbb{R})$, is immediate from Theorem \ref{KN1}.
\begin{corollary}\label{KN2}
 Let $g$ be as in the Theorem \ref{KN1}. Then $(g,a,b)$ is a Weyl-Heisenberg frame if and only if the Kohn-Nirenberg operator $K_{\Psi}$ is bounded below: there is a positive constant $\alpha$ such that $ \langle K_{\Psi}f, f\rangle \geq \alpha \|f\|_2^2,f\in L^2(\mathbb{R})$.
\end{corollary}
As a simple application of our methods, we look at the Walnut representation of the Weyl-Heisenberg frame operator $S$, presented in \cite{C} as
\\\hspace*{2cm}
 $ Sf(x) = (1/b) \underset{k \in \mathbb{Z}}{\Sigma} (T_{k/b} f)(x) G_{k}(x)$ for all $f\in L^2(\mathbb{R})$,
\\\noindent where the series is absolutely convergent for almost all $x\in \mathbb{R}$.
 A thorough discussion on this can also be found in \cite{CCJ}.
\begin{proposition}
Let $g \in \mathcal P_{a,b}$ generates a frame $(g, a, b )$ with $S$ as its frame operator. Then the representation of $S$ as the pseudo-differential operator with symbol $\Psi$ yields the Walnut representation and conversely.
 \end{proposition}
\begin{proof}
It is easy to see that
 $\mathcal F^{-1} M_{\Psi_{x}} \mathcal F f
 = (1/b) \Sigma G_{k}(x) T_{(k/b)} f$, using the absolutely convergent Fourier expansion of
$\Psi$ obtained in Proposition \ref{FSP}.
 If $f\in A_1(\mathbb{R})$, then the right side is uniformly convergent and so each side is
a continuous function. Evaluating
at $x$, we get the Walnut representation.

Conversely, from the Walnut representation we can get

 $ Sf(x) \ \ =((1/b)\Sigma  G_{k}(x)\mathcal F^{-1} \mathcal F T_{k/b}) f(x)$
\\\hspace*{1.5cm} $ =\mathcal F^{-1} ((1/b)\Sigma  G_{k}(x) \mathcal F T_{k/b}) f(x)$
\\\hspace*{1.5cm} $= \mathcal F^{-1} ((1/b)\Sigma  G_{k}(x) E_{k/b} \mathcal F)f(x)$
 \\\hspace*{1.5cm} $= \mathcal F^{-1} \Psi_{x}\mathcal F(f)(x)$, again by Proposition \ref{FSP}.
\end{proof}

 The series $Sf = (1/b) \Sigma_k (T_{k/b} f) G_{k}$ in the Walnut representation converges in norm for all $f\in L^2(\mathbb{R})$ whenever $g \in W$ and the operator is bounded in norm (see 6.3.2 of \cite{G}) in this case. The first part of the proof above shows that the same conclusions hold if $g \in \mathcal {P}_{a,b}$ satisfies conditions of Theorem \ref{KN1}.
The following observation highlights the significance of the symbol $\Psi$ in characterising Weyl-Heisenberg frames.
\begin{theorem}\label{CWF}
If $g \in \mathcal E_{a,b}$ is such that $\Psi(x,\xi)$ is a function of $\xi$, say $ \psi_2 (\xi) = \Psi (x,\xi)$ a.e., then $(g,a,b)$ is a frame if and only if a condition of the form
$0<\alpha \leq  \Psi(x, \xi) \leq \beta$ holds a.e. $x,\xi \in \mathbb{R}$.
\end{theorem}

\begin{proof}
Under the hypothesis,  $ \langle \mathcal {F}^{-1} M _{\Psi_{x}}\mathcal {F}f, f \rangle  = \langle \psi_{2}\widehat{f}, \widehat{f} \rangle, f \in L^{2}(\mathbb{R})$.
If $(g,a,b)$ is a frame with frame operator $S$, then
$\alpha \|f\|^{2} \leq \langle Sf, f \rangle \leq \beta \| f \|^{2}$, say,
for all $f \in L^{2}(\mathbb{R})$.
But $Sf(x) = \mathcal {F}^{-1}M_{\psi_{2}}\mathcal {F}f(x)$, so $S = \mathcal {F}^{-1}M_{\psi_{2}}\mathcal {F}$. Thus  $\alpha I \leq M_{\psi_2} \leq \beta I$
and so $\psi_2$ satisfies the asserted inequalities.

Conversely, suppose $0<\alpha \leq  \Psi(x,\xi) \leq \beta$ for a.e. $x,\xi \in \mathbb{R}$.
Since $\|f\|^{2} = \|\widehat{f}\|^{2}$, we have
$\alpha \|f\|^{2} \leq \langle \psi_{2}\widehat{f}, \widehat{f} \rangle \leq \beta \| f \|^{2}$  for all
$f \in L^{2}(\mathbb{R})$ and

\vspace*{2mm}$\mathcal {F}^{-1} M_{\psi_{2}}\mathcal {F}f(x) =
\Sigma_{m,n} \ \langle f, E_{mb}T_{na}g\rangle E_{mb}T_{na}g(x)$,

\vspace*{2mm}$\langle \psi_{2}\widehat{f}, \widehat{f} \rangle = \langle \mathcal {F}^{-1} M _{\psi_{2}}\mathcal {F}f, f \rangle
=  \Sigma_{m,n} \ | \langle f, E_{mb}T_{na}g\rangle |^{2}$.

\noindent These inequalities, in tandem, prove that $(g,a,b)$ is a frame.

It remains to get the series for $\psi_2$. By arguments analogous to those used in the proof of Proposition \ref{FSP},
$\Psi^{\xi}(x) = \Psi(x,\xi)$ has an absolutely convergent Fourier series
$(1/a) \Sigma_k \gamma_{k} e^{2 \pi i(k/a) x}$
where the coefficients are given by

$\gamma_k = \gamma_{k} (\xi) = \Sigma_m \
\check{\overline{g}}(\xi-mb) \check{\overline{g}}(\xi-mb+k/a)$.

\noindent If $\Psi$ is independent of $x$, this Fourier series reduces to the constant

$\gamma_0 = \psi_2(\xi) = (1/a)  \Sigma_m |\check{\overline{g}}(\xi - mb)|^{2}$.
\end{proof}
A similar observation can be made when $\Psi(x,\xi)$ is a function of $x$ alone and consequently, a version of the famous \textit{painless nonorthogonal expansion} (\cite{DGM}) can be given when the Gabor atom lies in $\mathcal {P} :=\cap \{\mathcal {P}_{a,b}: ab < 1\}$. Another version of this is also possible using Theorem \ref{CWF}, when the Gabor atom is in $\mathcal {E}$ and the support condition is imposed on the \textit{Fourier domain}.
\begin{corollary}
Suppose $g \in \mathcal E$ is such that the support of $\widehat{g}$ lies in the interval $ [0, L]$. Then for any $a,b$ with $ b \leq L <1/a$, the associated $\Psi$ is independent of the first variable $x$, $\Psi(x,\xi) = \psi_2(\xi)$  for almost all $x$ and for all $\xi$ for which the series for $\psi_2$ converges and the conclusions of the last theorem  hold.
\end{corollary}
\begin{proof} Observe that
$\Psi$ is bounded with respect to both $x$ and $\xi$ since

$| \Psi_(x,\xi)|
 \leq  \underset{n \in \mathbb{Z}}{\Sigma} \ | g(x-na)| \ \underset{m \in \mathbb{Z}}{\Sigma}\ | \check{\overline{g}}(\xi - mb)|$
 $\leq  AB.$
\\\noindent Since $\Psi$ is $a$-periodic in $x$, we can consider its Fourier series      as above. The assumed properties on the support of $\widehat{g}$ imply that $\gamma_{k}(\xi) = 0$ for
$k \neq 0$ as before and
$ \Psi(x,\xi) = \gamma_{0}(\xi) = (1/a)  \Sigma_m |
\check{\overline{g}}(\xi - mb)|^{2}$.
The last part is clear.
\end{proof}

\bibliographystyle{amsplain}

\end{document}